\documentclass[12pt]{article}
 \usepackage[english]{babel}
  \usepackage{amsmath,amsfonts,amssymb,amsthm}
  \usepackage{bbm}

\usepackage{paralist,subfigure}

\usepackage[latin1]{inputenc}
\usepackage[OT1]{fontenc}

\usepackage[active]{srcltx} %%allows you to jump from your editor directly into
\usepackage{amsfonts}
\usepackage{amsmath}
 \usepackage{psfrag}
  \usepackage{epsfig}

\def\R{\mathbb R}

\def\C{\mathbb C}

\def\supp{{\mathrm {supp}}}
\def\epsilon{\varepsilon}
\def\e{\varepsilon}

\let\e=\varepsilon
\let\vp=\varphi
\let\t=\tilde
\let\ol=\overline
\let\ul=\underline

\let\.=\cdot
\let\0=\emptyset
\let\mc=\mathcal

\def\solose{\ \Rightarrow\ }
\def\sse{\ \Leftrightarrow\ }

\def\O{\Omega}

\def\di{\displaystyle}

\newcommand{\su}[2]{\genfrac{}{}{0pt}{}{#1}{#2}}

\let\di=\displaystyle
\def\trait (#1) (#2) (#3){\vrule width #1pt height #2pt depth #3pt}

\newcommand{\beq}{\begin{equation}}
\newcommand{\eeq}{\end{equation}}
\newcommand{\baa}{\begin{array}}
\newcommand{\eaa}{\end{array}}
\newcommand{\ba}{\begin{eqnarray}}
\newcommand{\ea}{\end{eqnarray}}

\newtheorem{theorem}{Theorem}[section]

\newtheorem{proposition}[theorem]{Proposition}

\newtheorem{lemma}[theorem]{Lemma}

\theoremstyle{definition}

\newtheorem{remark}[theorem]{Remark}

\title{\bf Fisher-KPP propagation in the presence of a line: further effects}
 \author{Henri {\sc Berestycki}$^{\hbox{a }}$,  
 Jean-Michel {\sc Roquejoffre}$^{\hbox{b }}$, Luca {\sc Rossi}$^{\hbox{c }}$\\
 \footnotesize{$^{\hbox{a }}$ Ecole des Hautes Etudes en Sciences Sociales}\\
\footnotesize{ CAMS, 54, bd Raspail F-75270 Paris, France}\\
 \footnotesize{$^{\hbox{b }}$ Institut de Math\'ematiques de Toulouse,
 Universit\'e Paul Sabatier}\\
 \footnotesize{118 route de Narbonne, F-31062 Toulouse Cedex 4, France}\\
 \footnotesize{$^{\hbox{c }}$Universit\`a degli Studi di Padova}\\
 \footnotesize{Dipartimento di Matematica, Via Trieste, 63 -
 35121 Padova, Italy}\\
 }

\date{}

\begin{document}

\maketitle

\begin{abstract}
\noindent  This paper is a continuation of \cite{BRR2} where a new model of
biological invasions in the plane directed by a line was introduced.  Here we
include new features such as transport and reaction terms on the line. 
Their interaction with the pure diffusivity in the plane is quantified in terms
of enhancement of the propagation speed. We establish conditions that determine
whether the spreading speed exceeds the standard Fisher KPP invasion speed.
These conditions involve the ratio of the diffusivities on the line and in the
field, the transport term and the reactions. We derive the asymptotic behaviour
for large diffusions or large transports. We also  discuss the biological
interpretation of these findings. 
\end{abstract}

\numberwithin{equation}{section}


\medskip


\section{Introduction}

This paper is a sequel of \cite{BRR2} in which we introduced a new model to describe biological invasions in the
plane when a strong diffusion takes place on a line. The purpose of this model is to understand the effect of a line interacting with  a homogeneous environment.  This type of questions arises for instance in studying the propagation of diseases directed by roads \cite{Sig} or the movements of animal populations in the presence of pathways allowing for a more rapid movement. An example of the latter is provided by recent observations of movements of wolves  along seismic lines in Western Canada \cite{McK}.   In \cite{BRR2} we derived the asymptotic speed of spreading in 
the direction of the line. There we showed that  for low diffusion the line has no effect, whereas,
past a threshold, the line enhances global diffusion in the plane in the
direction of the line. Moreover, the propagation
velocity on the line increases indefinitely as the diffusivity on the line grows
to infinity.  

The goal of the present paper is to include new features in this model such as
transport and reaction or mortality on the road and to understand the resulting
new effects. Taking into account these new elements is important when
discussing propagation directed by roads or along water stream networks.
These require new developments that turn out to make more
transparent the case considered in \cite{BRR2} as well.  

 Consider the line $\{ (x,0)\,:\, x\in \R\}$ in
the plane $\R^2$; we will refer to the plane as ``the field" and the line as
``the road".  For a single species, we consider a system that combines the
density of this population in the field $v(x,y,t)$ and  the density on the line
$u(x,t)$. The main questions that we want to understand are the effects of a
transport term $q\partial_xu$ as well as a decay rate  $\rho\geq0$ on the road. 
An invasive species that can be carried by streams of water is an example of a
situation where such additional terms are required.
The transport and the decay rate are considered here to be uniform, that is, $q$
and $\rho$ are constant. 
Due to the symmetry of the problem, we can restrict our analysis to the upper
half-plane $\O:=\{(x,y)\ :\ x\in\R,\ y>0\}$.
The equations for $u$ and $v$ then read:
\begin{equation}
\label{Cauchy-rho}
\begin{cases}
\partial_t u-D \partial_{xx} u+q\partial_x u+\rho u= \nu v(x,0,t)-\mu u &
x\in\R,\
t>0\\
\partial_t v-d\Delta v=f(v) & (x,y)\in\O,\ t>0\\
-d\partial_y v(x,0,t)=\mu u(x,t)-  \nu v(x,0,t) & x\in\R,\ t>0.
\end{cases}
\end{equation}
Here, $d,D,\mu,\nu>0$, $\rho\geq 0$, $q\in\R$ and $f\in C^1([0,+\infty))$
satisfies the usual KPP type assumptions:
\begin{equation}\label{fKPP} 
f(0)=f(1)=0,\quad f>0\text{ in }(0,1),\quad f<0\text{ in
}(1,+\infty),\quad f(s)\leq f'(0)s\text{ for }s>0.
\end{equation} 
Note that transport and pure decay only occur on the road and the question is to  know how these interact with the diffusivity and growth in the field.
We combine the system with the initial condition 
$$
\begin{cases}
u|_{t=0}=u_0 & \text{in }\R\\
v|_{t=0}=v_0 & \text{in }\O,
\end{cases}
$$
where $u_0$, $v_0$ are always assumed to be nonnegative and continuous.

Let $c_K$ denote the classical KPP invasion speed \cite{KPP} in the field:
$$
c_K=2\sqrt{df'(0)}.
$$
This is the asymptotic speed at which the population would spread in the open
space - i.e., when no line is
present (see \cite{AW}).

We say that \eqref{Cauchy-rho} admits 
the {\em asymptotic speeds of spreading} $w_*^\pm$ (in the directions
$\pm e_1=\pm(1,0)$)
if for any solution $(u,v)$ starting from a compactly
supported initial datum $(u_0,v_0)\not\equiv(0,0)$, and for all $\e>0$, the
following hold true:
% We actually show that solutions spread with velocity $w_*^\pm$ in a stronger
% sense. Namely, we prove that solutions emerging from nontrivial, compactly
% supported initial data satisfy
$$\lim_{t\to+\infty}\sup_{\su{x<-(w_*^-+\e)t}{y\geq0}}
|(u(x,t),v(x,y,t))|=0,\qquad\lim_{t\to+\infty}\sup_{\su{x>(w_*^++\e)t}{y\geq0}}
|(u(x,t),v(x,y,t))|=0,$$
$$\forall a>0,\quad\
\lim_{t\to+\infty}\sup_{\su{-(w_*^--\e)t<x<(w_*^+-\e)t}{
0\leq y<a}}|(u(x,t),v(x,y,t))-(U,V(y))|=0,$$
where $(U,V)=(U,V(y))$
is the unique positive, bounded, stationary solution of \eqref{Cauchy-rho}.

Thus, the first step for proving the existence of the asymptotic speeds of
spreading consists in deriving a uniqueness result for the stationary system.
We call it a Liouville-type result.
% we prove that the limits in the definitions of the asymptotic
% speeds hold uniformly in $c$ outside any neighbourhood of $w_*^\pm$ (see
% Remark
% \ref{rem:cuniform}).
The existence of the asymptotic speeds of spreading implies in particular
that
% 
% 
% Let us mention that the notion of asymptotic speed can sometimes be found in
% the
% literature in the following weaker formulation:
$$\forall c\notin[-w_*^-,w_*^+],\quad 
\lim_{t\to+\infty}(u(x+ct,t),v(x+ct,y,t))=(0,0),$$
$$\forall c\in(-w_*^-,w_*^+),\quad 
\lim_{t\to+\infty}(u(x+ct,t),v(x+ct,y,t))=(U,V(y)),$$
locally uniformly in $(x,y)\in\ol\Omega$. 
This kind of weaker formulation is sometimes used in the literature as the
definition of the asymptotic speed.

%%%%%%%%%%%%%%%%%%%%%%%%%%%%%%%%%%%%%%%%%%%%

\subsection{Statement of the main results}
\label{sec:main}

In \cite{BRR2}, we proved that, in the case $q,\rho=0$, \eqref{Cauchy-rho}
admits
asymptotic speeds of spreading $w_*^\pm$. We further
identified a threshold situation: $D/d=2$ 
below which $w_*^\pm=c_K$ and above which $w_*^+=w_*^->c_K$.
The first question we address in the present paper is how this threshold is
modified
by the presence of the transport and decay terms $q$ and $\rho$. This question
is solved by the following

\begin{theorem}\label{thm:rho}
Under the assumption \eqref{fKPP}, problem \eqref{Cauchy-rho} admits 
asymptotic speeds of spreading $w_*^\pm$ (in the directions $\pm e_1$).
Moreover, if $\di\frac{D}d\leq2+\frac\rho{f'(0)}\mp\frac q{\sqrt{df'(0)}}$,  
then $w^\pm_*=c_K$, else $w_*^\pm>c_K$.
\end{theorem}

In the case $q=\rho=0$, we recover the result of \cite{BRR2}. Theorem
\ref{thm:rho} shows that a mortality term $\rho>0$ always rises the threshold
for $D/d$ after which the effect of the road is felt. The threshold
for the enhancement of the speed towards right, $w_*^+$, is decreased if the
transport term $q$ is positive and increased if it is negative.
\\

We will actually carry out the whole study of system \eqref{Cauchy-rho} in the
case where the mortality term $-\rho u$ is replaced by a more general reaction
term $g(u)$. This term can be used to model situations where reproduction
occurs on the road as well. The general system reads
\begin{equation}\label{Cauchy}
\begin{cases}
\partial_t u-D \partial_{xx} u+q\partial_x u= \nu v(x,0,t)-\mu u+g(u) & x\in\R,\
t>0\\
\partial_t v-d\Delta v=f(v) & (x,y)\in\O,\ t>0\\
-d\partial_y v(x,0,t)=\mu u(x,t)-  \nu v(x,0,t) & x\in\R,\ t>0.
\end{cases}
\end{equation}
The hypotheses are: $f,g\in C^1([0,+\infty))$ and satisfy
\begin{equation}\label{hyp:f}
f(0)=f(1)=0,\qquad f>0\text{ in }(0,1),\qquad f<0\text{ in }(1,+\infty),
\end{equation}
\begin{equation}\label{hyp:g}
g(0)=0,\qquad \exists S>0,\ g(S)\leq0,
\end{equation}
\begin{equation}\label{concave}
s\mapsto f(s)/s,\quad s\mapsto g(s)/s\quad\text{are nonincreasing}. 
\end{equation}
Condition \eqref{concave} holds if $f$ and $g$ are concave. It implies that 
$f$ and $g$ are of KPP type:
$f(s)\leq f'(0)s$ and $g(s)\leq g'(0)s$. 
% Our main result is the

\begin{theorem}\label{thm:main}
Assume that \eqref{hyp:f}-\eqref{concave} hold. Then:
\begin{enumerate}[(i)]
 \item {\em(Liouville-type result).}
Problem \eqref{Cauchy} admits a unique, positive, bounded, 
stationary solution $(U,V)$. Moreover, $U\equiv\text{constant}$
and $V\equiv V(y)$.
\item {\em(Spreading).} 
Problem \eqref{Cauchy} admits asymptotic speeds of spreading $w_*^\pm$. 
% (in the directions $\pm e_1$).
\item {\em(Spreading velocity).} 
If $\di\frac{D}d\leq2-\frac{g'(0)}{f'(0)}\mp\frac q{\sqrt{df'(0)}}$,  
then $w^\pm_*=c_K$. Otherwise $w_*^\pm>c_K$.
\end{enumerate}
\end{theorem}
Notice that, in Theorem \ref{thm:rho}, $s\mapsto f(s)/s$ is not assumed to be
nonincreasing. This hypothesis is only required in the proof of statement (i) of
Theorem \ref{thm:main}. An alternative hypothesis is given in
Proposition \ref{pro:Liouville2} below. It is to be noted that the
Liouville-type
result may fail if $s\mapsto f(s)/s$ is not nonincreasing. In this case, we 
still get a convergence result, but for initial
data satisfying $u_0\leq\nu/\mu$, $v_0\leq1$, and the convergence
holds to the minimal positive, stationary solution (see Remark
\ref{rem:noLiouville}). Furthermore, the
spreading speeds are still well defined and statements (ii) and
(iii) of Theorem \ref{thm:main} hold true. 
% These issues are discussed at the end of Section \ref{sec:Liouville}.

When there is no transport on the road,
$w^-_*$ and $w^+_*$ coincide and the threshold condition given by
Theorem~\ref{thm:rho} part
(iii) becomes
$$
\frac{D}d\leq2-\frac{g'(0)}{f'(0)}.
$$
% The biological interpretation is less obvious, but it 
This allows us to understand the - somewhat mysterious - factor 2 in the
threshold condition of \cite{BRR2}. Indeed, when $g\equiv0$, we recover the
condition $D\leq2d$ of \cite{BRR2}. We further note that, if $f'(0)=g'(0)$, the
threshold
condition
becomes $D=d$ - which is what one would have expected. Thus, when the same reaction 
occurs on the road and in the field, the effect of the road is felt as
soon as the diffusivity there is larger than in the field. The factor $2$ of
\cite{BRR2} is therefore explained by the absence of reaction on the road.
\\

In the last part of the paper, we investigate the limits of the asymptotic
speeds of spreading $w_*^\pm$ as the diffusion and the transport on the road
tend to $+\infty$. We find the following asymptotic behaviours:
\begin{theorem}\label{thm:limits}
As functions of the variables $D$ and $q$ respectively,
$w_*^\pm$ satisfy
$$\lim_{D\to\infty}\frac{w_*^\pm}{\sqrt D}=h,\qquad
\lim_{q\to\pm\infty}\frac{w_*^\pm}{|q|}=\begin{cases}
                                   k & \text{if }g'(0)<\mu\\
                                   1 & \text{if }g'(0)\geq\mu,\\
                                  \end{cases}
$$
with $h>0$ independent of $q$ and $0<k<1$ independent of $D$.
\end{theorem}
The expression of $k$ is explicitly given in the proof of the theorem 
in Section \ref{sec:largeq}.

%%%%%%%%%%%%%%%%%%%%%%%%%%%%%%%%%%%%%%%%%%%%

\subsection{Biological interpretation}
\label{sec:bio}

What we are aiming at understanding here is how transport,
diffusivity
and reaction on the road combine to yield a spreading speed larger than
$c_K$, the KPP invasion speed in the open field (without the road). 
% If this occurs we say that the road has an effect.
For definiteness, let us consider propagation to the right (that is, in the
direction $e_1$).

In the case when the reaction on the road consists in a pure mortality term
$-\rho u$, Theorem \ref{thm:rho} asserts that the threshold for
$D/d$ after which the effect of the road is felt grows linearly in $\rho$.
This means that the larger the $\rho$, the more likely the road has no
effect on the overall propagation.
% 
% On the contrary, when $q<0$, the overall
% spreading to the right is hampered.  
% In the absence of mortality on the road, that is 
% Let us discuss the influence of the transport term $q$ on the propagation, say,
% towards right. 
When $\rho=0$, the threshold condition reads
$$\di\frac{D}d>2-\di\frac q{\sqrt{df'(0)}}=2\left(1-\frac{q}{c_K}\right).
$$
 The above formula  can also be written as
$$
 q > c_K\left(1 - \frac{D}{2d}\right).
$$
A consequence of this formula is that a transport $q$ larger than
$c_K$ is sufficient to enhance the overall propagation, no matter what the
diffusivity ratio is. This explains for instance how a river may help the
invasion of a parasite
\cite{Plant}. It may seem at first sight a natural effect that a drift term
larger than $c_K$ on the road will enhance the overal invasion speed and bring
it above $c_k$. However, it should be noted that in the situation we are looking
at, on the road alone - without the exchange terms with the field - there would
be extinction of the population even with a large $q$. The global spreading
speed results from a rather delicate interaction between the field and the road.
Therefore, it is remarkable that the threshold is precisely $q=c_K$.
If $q\leq c_K(1 -D/2d)$ then the spreading is with the usual
KPP velocity. This yields an interesting interpretation when considering the
propagation against the direction of the transport, that is, when $q<0$. For
$-q$ large enough, the asymptotic speed of spreading in the direction $e_1$ in
the field is the KPP invasion speed. For instance, asymptotic propagation
upstream against a
river flow in the neighbouring field is unaffected by the river, but downstream
propagation, in the direction of the flow, can be significantly enhanced.

Theorem \ref{thm:limits} asserts that the spreading speed approaches a
portion of the speed $q$ of the flow when the latter is very large. This portion
is equal to $1$ only if there is a sufficiently large reaction $g$ on the road.
It is worthwhile to note that the condition on $g$ only involves the rate
$\mu$ at which the individuals leave the line.

Another phenomenon for which our model is relevant is the dynamics of a
population that favours disturbed habitats\footnote{The authors thank Mark Lewis
for having brought this issue to their attention.}.
It is known that certain invasive plants, some weeds in particular, thrive in
disturbed habitats, such as cultivated fields or along roads
\cite{disturbance1}.  It has been observed that the presence of
disturbances increases their speed of spreading \cite{disturbance2}. 
A model has been proposed in [2] to describe this type of phenomena and in
particular the propagation of the ``unscented chamomile'' in North America. This
is an invasive weed widely distributed in croplands, pastures and infrastructure
(road edges). 
In such cases, the disturbance, represented by the road, provides better
environmental conditions than the rest of the territory
(more diffuse light, bare ground that permit seedling establishment), rather
than a greater diffusivity. This
results in having $d=D$ and $g>f$ in our model.
If one actually has $g'(0)>f'(0)$ then Theorem \ref{thm:main} implies that
there is enhancement of the invasion speed, which is in agreement with the
observations.
Notice that if $g'(0)\geq2f'(0)$ then the enhancement always
occurs, whatever the diffusivity ratio is. We plan to study the effect due to
propagation through a road crossing a globally unfavourable area in a separate
note.

%%%%%%%%%%%%%%%%%%%%%%%%%%%%%%%%%%%%%%%%%%%%

\subsection{Organization of the paper}
\label{sec:plan}

The paper is organized as follows.
In the next section, we start with recalling the results of \cite{BRR2}
concerning the well-posedness of the Cauchy problems and the comparison
principles. 
% We then study the steady states of \eqref{Cauchy}; although
% the result is the same as in \cite{BRR2}, the proof needs some adaptations. 
Then, we show that the large time limits of solutions emerging
from non-trivial, bounded initial data are trapped
between two positive, bounded, stationary solutions which do not depend on $x$.
In Section \ref{sec:Liouville}, we derive the Liouville-type results for such
class of solutions, first for \eqref{Cauchy-rho} and then for \eqref{Cauchy}.
The large-time behaviour of solutions is thereby characterized, at least
locally uniformly in space.
In particular, Theorem \ref{thm:main} part (i) follows.

In section \ref{sec:velocity}, we investigate the
spreading property of solutions, that is, the set where they converge to $0$
or to the positive, stationary solution.
We give the proof of Theorem \ref{thm:main} parts (ii) and (iii) insisting on
the main differences with \cite{BRR2}.

Section \ref{sec:limits} is dedicated to the proof of Theorem \ref{thm:limits}.

%%%%%%%%%%%%%%%%%%%%%%%%%%%%%%%%%%%%%%%%%%%%

\section{Long time behaviour}
\label{sec:ltb}

Throughout this section, we assume that \eqref{hyp:f}, \eqref{hyp:g} hold. The
unique solvability for the Cauchy problem associated with \eqref{Cauchy}
in the class of nonnegative, bounded functions follows from the same arguments
as in \cite{BRR2}, Proposition 3.1. There, the case $g\equiv0$ is treated, but
a slight modification of the arguments allows one to handle the presence of the
nonlinear term $g$. The existence result is obtained by first solving the Cauchy
problem for $u$ with $v\equiv0$ and then the one for $v$ with the
so obtained $u$. Repeating this procedure, starting with the new $v$, one
constructs an increasing, bounded sequence $(u_n,v_n)$ which eventually
converges to a solution of \eqref{Cauchy}.
The uniqueness result is a
consequence of the comparison principle between  sub and supersolutions (i.e.,
pairs of functions satisfying the system with the ``$=$'' replaced by ``$\leq$''
and ``$\geq$'' respectively). We recall two versions of the comparison
principle that will be used several time in the sequel:
\begin{proposition}[\cite{BRR2}]\label{comparison}
Let $(\underline u,\underline v)$ and $(\overline
u,\overline v)$ be respectively a 
subsolution bounded from above and a supersolution bounded
from below of
\eqref{Cauchy} satisfying $\underline u\leq
\overline u$ and $\underline v\leq
\overline v$ at $t=0$. Then, either $\underline u<\overline u$ and
$\underline v<\overline v$ for all $t$, or there exists $T>0$ such that 
$(\underline u,\underline v)=(\overline u,\overline v)$ for $t\leq T$.
\end{proposition}
The second version deals with subsolutions in a generalised sense (see
\cite{BL80} for a related notion).

\begin{proposition}[\cite{BRR2}]\label{gensub}
Let $E\subset\R^N\times(0,+\infty)$ and $F\subset\O\times
(0,+\infty)$ be two open sets and let $(u_1,v_1)$, $(u_2,v_2)$ be two
subsolutions of \eqref{Cauchy} bounded from above and satisfying
$$u_1\leq u_2\quad\text{on }(\partial E)\cap(\R^N\times(0,+\infty)),\qquad
v_1\leq v_2\text{ on }(\partial F)\cap(\O\times(0,+\infty)).$$
If the functions $\underline u$, $\underline v$ defined by
$$\underline u(x,t):=\begin{cases}
                      \max(u_1(x,t),u_2(x,t)) & \text{if }(x,t)\in\overline E\\
u_2(x,t) & \text{otherwise},
                     \end{cases}$$
$$\underline v(x,y,t):=\begin{cases}
                      \max(v_1(x,y,t),v_2(x,y,t)) & \text{if }(x,y,t)\in
\overline F\\
v_2(x,y,t) & \text{otherwise},
                     \end{cases}$$
satisfy
$$\underline u(x,t)>u_2(x,t)\ \Rightarrow\ \underline v(x,0,t)\geq v_1(x,0,t),$$
$$\underline v(x,0,t)>v_2(x,0,t)\ \Rightarrow\ \underline u(x,t)\geq u_1(x,t),$$
%$$\underline u(x,t)<u_1(x,t)\Rightarrow \underline v(x,0,t)=v_2(x,0,t),$$
%$$\underline v(x,0,t)<v_1(x,0,t)\Rightarrow \underline u(x,t)=u_2(x,t),$$
then, any supersolution $(\overline
u,\overline v)$ of \eqref{Cauchy} bounded from below and such that $\underline
u\leq
\overline u$ and $\underline v\leq
\overline v$ at $t=0$, satisfies $\underline u\leq\overline u$ and $\underline
v\leq\overline v$ for all $t>0$.
\end{proposition}

We now derive a result which gives a preliminary information about the long time
behaviour of solutions.

\begin{lemma}\label{lem:trapped}
Let $(u,v)$ be the solution of \eqref{Cauchy} starting from a bounded
initial datum $(u_0,v_0)\not\equiv(0,0)$. Then, there
exist two positive, bounded, $x$-independent, stationary solutions $(U_1,V_1)$
and $(U_2,V_2)$ of
\eqref{Cauchy} such that
$$U_1\leq \liminf_{t\to+\infty}u(t,x)\leq\limsup_{t\to+\infty}u(t,x)\leq U_2,$$
$$V_1(y)\leq\liminf_{t\to+\infty}v(t,x,y)\leq
\limsup_{t\to+\infty}v(t,x,y)\leq V_2(y),$$
locally uniformly in $(x,y)\in\ol\Omega$.
\end{lemma}

\begin{proof}
Let $S$ be the constant in \eqref{hyp:g}. The pair $(\ol U,\ol V)$ defined by
$$(\ol U,\ol
V)=\left[\max\left(\frac{\|u_0\|_\infty+S}{\nu},\frac{\|v_0\|_\infty+1}{\mu}
\right)\right](\nu , \mu),$$
is a supersolution of \eqref{Cauchy} which is larger than $(u,v)$ at $t=0$.
Let $(\ol u,\ol v)$ be the solution of \eqref{Cauchy} with initial datum
$(\ol U,\ol V)$. Using the comparison principle given by Proposition
\ref{comparison}, we see that $\ol u$ and $\ol v$ are nonincreasing in $t$.
Thus, $(\ol u,\ol v)$ converges as $t\to+\infty$ to a
stationary solution
$(U_2,V_2)$ of \eqref{Cauchy} satisfying
$$\limsup_{t\to+\infty}u(t,x)\leq U_2(x),\qquad 
\limsup_{t\to+\infty}v(t,x,y)\leq V_2(x,y),$$
locally uniformly in $(x,y)\in\ol\Omega$.
Furthermore, by translation invariance of the problem in the $x$-direction,
we see that $(U_2,V_2)$ does not depend on $x$.

We now construct the pair $(U_1,V_1)$. Take $R>0$ large enough
in such a way that 
the principal eigenvalue of $-\Delta$ in $B_R\subset\R^2$ with Dirichlet
boundary condition is less than $f'(0)/(2d)$. The associated principal
eigenfunction $\vp_R$ satisfies
$$-d\Delta\vp_R\leq \frac12 f'(0)\vp_R\quad\text{in }B_R.$$
Hence, for $\e>0$ small enough, the function $\e\vp_R$ satisfies 
$-d\Delta(\e\vp_R)\leq f(\e\vp_R)$ in $B_R$. We extend $\vp_R$ to $0$ outside
$B_R$ and we define $\ul V(x,y):=\e\vp_R(x,y-R-1)$. Thus, $(0,\ul V)$ is a
generalised subsolution of \eqref{Cauchy}. The strong comparison principle
given by \ref{comparison} implies that $u$ and $v$ are positive for $t>0$.
Hence, up to decreasing $\e$ if need be, we have that $(0,\ul V)$ is below
$(u,v)$ at, say, $t=1$. Let $(\ul u,\ul v)$
be the solution of \eqref{Cauchy} starting from
$(0,\underline V)$ at $t=1$. Using the comparison principle for
generalised subsolutions - Proposition \ref{gensub} - we see that
$\ul u$ and $\ul v$ are nondecreasing in $t$. If one of them were not strictly
increasing, the strong comparison principle of Proposition \ref{comparison}
would
imply that $(\ul u,\ul v)$ is constant in time, which is impossible because
$(0,\ul V)$ is not a solution of \eqref{Cauchy}. Thus, as $t\to+\infty$,
$(\ul u,\ul v)$ converges to a stationary solution
$(U_1,V_1)$ of \eqref{Cauchy} satisfying
$$0<U_1(x)\leq\liminf_{t\to+\infty}u(t,x),\qquad \ul
V(x,y)<V_1(x,y)\leq\liminf_{t\to+\infty}v(t,x,y),$$
locally uniformly in $(x,y)\in\ol\Omega$.
It remains to show that $(U_1,V_1)$ does not depend on $x$.
Since $\ul V$ is compactly supported, there exists $k>0$ such that
$(U_1,V_1)$ is above the translated
by any $h\in(-k,k)$ in the $x$-direction of $(0,\ul V)$.
By translation invariance of the problem, the solutions of \eqref{Cauchy}
emerging from these initial data coincide with the translated by $h\in(-k,k)$ of
$(\ul u,\ul v)$. We then infer, by comparison, that
$(U_1,V_1)$ is above the translated by $h\in(-k,k)$ in the $x$-direction of
itself. Namely, it does
not depend on $x$.
\end{proof}

%%%%%%%%%%%%%%%%%%%%%%%%%%%%%%%%%%%%%%%%%%%%

\section{Liouville-type result for $1$-dimensional solutions}
\label{sec:Liouville}

In this section, we derive a Liouville-type result for stationary
solutions of \eqref{Cauchy} which do not depend on $x$. Namely, we will show
that the problem 
\begin{equation}\label{y-stationary}
\begin{cases}
U\equiv\text{constant, }V\equiv V(y)\\
-d V''=f(V), & y>0\\
\nu V(0)=\mu U-g(U)\\
-d V'(0)=g(U).
\end{cases}
\end{equation}
admits a unique positive, bounded solution.
This will imply the general Liouville-type result, Theorem \ref{thm:main}
part (i), because, by Lemma \ref{lem:trapped}, any positive, bounded, stationary
solution of \eqref{Cauchy} lies between two positive, bounded
solutions of \eqref{y-stationary}.

We start with considering the pure mortality case 
$g(U)=-\rho U$. The proof is much simpler in this case. Problem 
\eqref{y-stationary} reduces to
\begin{equation}\label{y-mortality}
\begin{cases}
U\equiv\frac{\nu}{\mu+\rho} V(0)\\
-d V''=f(V), & y>0\\
d V'(0)=\frac{\nu\rho}{\mu+\rho} V(0).
\end{cases}
\end{equation}
\begin{proposition}\label{pro:y}
Under the assumption \eqref{hyp:f}, problem
\eqref{y-mortality} admits a unique positive, bounded solution. 
\end{proposition}

\begin{proof}
Let $V$ be a positive, bounded solution of \eqref{y-mortality}.
It is straightforward to check that $V$ necessarily satisfies $0<V\leq 1$,
$V(+\infty)=1$ and $V'(+\infty)=0$. 
Thus, multiplying the second equation of
\eqref{y-mortality} by $V'$ and integrating by parts between $0$ and $+\infty$
we get
$$\int_{V(0)}^1 f(s)ds=\frac d2 (V'(0))^2=\frac{\nu^2\rho^2}{2d(\mu+\rho)^2}
V^2(0).$$
Examining the function $\theta$ defined by
$$\theta(\sigma):=\frac{\nu^2\rho^2}{2d(\mu+\rho)^2}
\sigma^2-\int_\sigma^1 f(s)ds,$$
we see that 
$$\theta(0)<0,\qquad\theta(1)\geq0,\qquad\theta'>0\text{ in }(0,1).$$
There exists then a unique value $\sigma_0\in(0,1]$ such that
$\theta(\sigma_0)=0$. Hence, $V(0)=\sigma_0$. This proves the uniqueness of
positive, bounded solutions of \eqref{y-mortality}. It is also easy to verify
that the solution $V$ of \eqref{y-mortality} with initial datum $V(0)=\sigma_0$
is actually positive and bounded. Assume indeed by contradiction that this is 
not the case. One can then find
$\xi\geq0$ such that either $V(\xi)=1$ and $V'(\xi)>0$, or $V(\xi)<1$ and
$V'(\xi)=0$.
Owing to \eqref{hyp:f}, both cases are ruled out by the following equality:
$$\int_{\sigma_0}^{V(\xi)} f(s)ds=\frac d2 [(V'(0))^2-(V'(\xi))^2]=
\int_{\sigma_0}^1 f(s)ds-\frac d2(V'(\xi))^2.$$
\end{proof}
This concludes the proof of the Liouville-type result for stationary
solutions of \eqref{Cauchy-rho}.
% %Combining Lemma \ref{lem:trapped} and Proposition \ref{pro:y} we derive
% \begin{corollary}\label{cor:Liouville}
% Under the assumption \eqref{hyp:f}, system
% \eqref{Cauchy-rho} admits a unique positive, bounded, stationary solution
% $(U,V)$. Moreover, $U\equiv\text{constant}$ and $V\equiv V(y)$.
% \end{corollary}

We now pass to the case of a general reaction term on the road.

\begin{theorem}\label{thm:y}
Under the assumptions \eqref{hyp:f}-\eqref{concave}, problem
\eqref{y-stationary} admits a unique positive, bounded solution. 
\end{theorem}

The existence result is
contained in Lemma \ref{lem:trapped}, that we proved using a sub
and
supersolution argument.
Let us present a more explicit construction inspired by the shooting
method.
% , which also yields as a byproduct the additional properties stated in
% the lemma.

\begin{proof}[Proof of the existence part of Theorem \ref{thm:y}]
For $U\in\R$, let $V_U$ denote the associated solution of
\eqref{y-stationary}. Consider the following
set:
$$\mathcal{U}:=\{U>0\ :\ \forall y\geq0,\ V_U(y)>0\}.$$
This set is nonempty because, for $U\geq\max(\nu/\mu,S)$ ($S$ being 
the constant in \eqref{hyp:g}) the function
$V_U$
satisfies $V_U(0)\geq1$, $V_U'(0)\geq0$.
It follows that $V_U$ is nondecreasing and then positive. We then define
$$U^*:=\inf\mathcal{U}.$$
Suppose by contradiction that $U^*=0$. We can then take $U\in\mc{U}$ close
enough to $0$ in such a way that $V_U(0)<1$.
Notice that $V_U'(0)>0$, because otherwise $V_U$ would not be positive.
Call $\eta$ the first point where $V_U$ reaches the value $1$, if it exists,
else set $\eta:=+\infty$. In the second case, the function $V_U$ lies in
$(0,1)$ and then, since $dV_U''=-f(V_U)<0$, it satisfies $V_U(+\infty)=1$,
$V_U'(+\infty)=0$. 
Hence, in both cases, $V_U'(\eta)\geq0$.
We derive
$$V_U'(0)\geq\frac1d\int_0^\eta f(V_U(y))dy
\geq\frac1d\int_0^\eta\frac{V_U'(y)}{V_U'(0)}f(V_U(y))dy=
\frac{\int_{V_U(0)}^1 f(s)ds}{dV_U'(0)}.$$
Choosing $U\in\mc{U}$ small enough then leads to a contradiction, because
$V_U(0),V_U'(0)\to0$ as $U\to0$.
% 
% If $0<U<\min(\nu/\mu,1)$, then $V_U(0)<1$ and $V_U'(0)\leq0$, which
% implies that $V_U$ vanishes somewhere. Thus, $U^*\geq\min(\nu/\mu,1)$.
Therefore, $U^*>0$.

We claim that $V_{U^*}$ is positive and bounded.
The continuity of solutions of the Cauchy problem with respect to the initial
datum yields $V_{U^*}\geq0$.
Assume by way of contradiction that there exists $y_0>0$ such that
$V_{U^*}(y_0)=0$. We necessarily have that $V_{U^*}'(y_0)=0$, because otherwise 
$V_{U^*}$ would be negative somewhere. Hence, $V_{U^*}\equiv0$ by the uniqueness
of solutions of the Cauchy problem. But then $U^*=0$, which is impossible. If
$V_{U^*}(0)=0$ then 
$$V_{U^*}'(0)=-g(U^*)/d=-\mu U^*/d<0,$$
which is again a contradiction. We have shown that $V_{U^*}>0$ on $[0,+\infty)$.

Assume now by contradiction that $V_{U^*}$ is not bounded from above.
There exists then $y_0\geq0$ such that
% for some $y_0\geq0$. If $V_{U^*}'(y_0)\leq0$ then $-d V''=f(V)$ would yield
% $V_{U^*}'(0)\leq0$, which is impossible because $V_{U^*}'(0)=-g(U^*)/d$ and
% $U^*>1$. Therefore,
% $V_{U^*}'(y_0)>0$ and then there exists $y_1>y_0$ such that 
$$V_{U^*}(y_0)>1,\qquad V_{U^*}'(y_0)>0.$$
Since $V_U\to V_{U^*}$ as $U\to U^*$ in $C^1_{loc}([0,+\infty))$, for $U$
close enough to $U^*$ we have that
$$\min_{[0,y_0]}V_U>\frac12\min_{[0,y_0]}V_{U^*}>0,\qquad V_U(y_0)>1,
\qquad V_U'(y_0)>0.$$
It follows that $V_U$ is increasing in $[y_0,+\infty)$ and then $\min V_U>0$.
This means that $U\in\mathcal{U}$, which contradicts the definition of $U^*$.
Therefore, $(U^*,V_{U^*})$ is a positive, bounded solution of
\eqref{y-stationary}.
% 
% Case $\mu>\nu$.\\
% If $U\geq1$ then $V_U(0)>1$ and $V_U'(0)\geq0$. It follows that
% $V_U(y)\to+\infty$ as
% $y\to+\infty$ and then, in particular, that $U\in\mc{U}$.
% On the other hand, if $0\leq U\leq\nu/\mu$ then $V_U(0)<1$ and
% $V_U'(0)\leq0$. 
% This implies that $V_U$ vanishes somewhere, that is, $U\notin\mc{U}$.
% As a consequence
% $$[1,+\infty)\subset\mc{U}\subset(\nu/\mu,+\infty).$$
% As seen before, the continuity of solutions of the
% Cauchy problem with respect to the initial datum yields $U^*\in\mc{U}$.
% Also, if $V_{U^*}(y)\to+\infty$ as $y\to+\infty$ we would have that 
% $U\in\mc{U}$ for $U$ close enough to $U^*$, which is impossible.
% Therefore $(U^*,V_{U^*})$ is a positive bounded solution of
% \eqref{y-stationary}.
% This also shows that $U^*<1$, because $V_1(y)\to+\infty$ as $y\to+\infty$.
% It remains to show that $V_{U^*}>1$. Assume by way of contradiction 
% that there exists $y_0\geq0$ such that $V_{U^*}(y_0)\leq1$.
% We necessarily have $V_{U^*}'(y_0)<0$, because otherwise we would find
% $0\leq V_{U^*}'(0)=-g(U^*)$, which is impossible because
% $\mu/\nu\leq U^*<1$. But then $V_{U^*}$ vanishes somewhere, which 
% contradicts $U^*\in\mc{U}$.
% This concludes the proof.
\end{proof}

\begin{proof}[Proof of the uniqueness part of Theorem \ref{thm:y}]
Assume by way of contradiction that \eqref{y-stationary} admits two distinct
positive, bounded solutions 
$(U_1,V_1)$ and $(U_2,V_2)$. The uniqueness of solutions of the Cauchy problem
yields
$U_1\neq U_2$; say, $U_1<U_2$.
Since \eqref{concave} implies that the function $G:\R_+\to\R$ defined by
$G(s):=\mu-g(s)/s$ is nondecreasing, we obtain
$$\nu V_1(0)=G(U_1)U_1\leq G(U_2)U_1\leq\nu V_2(0),$$
and equality holds if and only if $V_2(0)=0$. But if $V_2(0)=0$ then
$-dV_2'(0)=\mu U_2>0$ and then $V_2$ would not be nonnegative. Thus,
$V_1(0)<V_2(0)$.
We argue differently depending on the fact that $V_1=V_2$ somewhere or not.

Case 1) $V_1(y)=V_2(y)$ for some $y>0$.\\
Let $\eta$ be the smallest $y>0$ such that $V_1(y)=V_2(y)$. Namely, $V_1<V_2$ in
$(0,\eta)$ and $V_1(\eta)=V_2(\eta)$. By the uniqueness of solutions of the
Cauchy problem we infer that $V_1'(\eta)>V_2'(\eta)$.
Multiplying by $V_2$ the equation satisfied by $V_1''$ and by $V_1$ the one
satisfied by $V_2''$ and integrating over $(0,\eta)$, we get
\[\begin{split}
\frac1d\int_0^\eta V_1V_2\left(
\frac{f(V_1)}{V_1}-\frac{f(V_2)}{V_2}\right) &=
\int_0^\eta (V_2''V_1-V_1''V_2)=
\big[V_2'V_1-V_1'V_2\big]_0^\eta\\
&<V_1'(0)V_2(0)-V_2'(0)V_1(0).
\end{split}\]
The first integral above is nonnegative by \eqref{concave}. Thus,
\begin{equation}\label{V12(0)}
V_1'(0)V_2(0)>V_2'(0)V_1(0).
\end{equation}
This inequality reads
$$g(U_1)(\mu U_2-g(U_2))<g(U_2)(\mu U_1-g(U_1)),$$
that is,
$$g(U_1)U_2<g(U_2)U_1.$$
This contradicts \eqref{concave} because $U_1<U_2$.

Case 2) $V_1(y)<V_2(y)$ for all $y>0$.\\
% Notice that $V_1$ and $V_2$ cannot attain the value $1$ without being
%  constant, 
% because otherwise they would be either unbounded or negative somewhere. The
% same consideration shows that 
Notice that $V_1$ and $V_2$ cannot attain
the value $1$ with a nonzero first derivative, 
because they would be either unbounded or negative somewhere. It follows
from the uniqueness of solutions of the Cauchy problem
that they are either identically
equal to $1$, or below $1$, concave and increasing or above $1$,  and
decreasing. In any case they satisfy
$$\forall i\in\{1,2\},\quad
\lim_{y\to+\infty}V_i(y)=1,\quad\lim_{y\to+\infty}V_i'(y)=0.$$
% As shown in the proof of the existence part, any solution $(U,V)$ of
% \eqref{y-stationary} satisfies $V'(+\infty)=0$.
Thus, the same computation as in the case 1 yields
\[\begin{split}
\frac1d\lim_{\eta\to+\infty}\int_0^\eta V_1V_2\left(
\frac{f(V_1)}{V_1}-\frac{f(V_2)}{V_2}\right) &=
\lim_{\eta\to+\infty}\int_0^\eta (V_2''V_1-V_1''V_2)\\
&=V_1'(0)V_2(0)-V_2'(0)V_1(0).
\end{split}\]
We know from \eqref{concave} that the function $s\mapsto f(s)/s$ is
nonincreasing. 
Moreover, since it is positive in $(0,1)$ and negative in $(1,+\infty)$, 
it cannot be constant in a left or in a right neighborhood of $1$.
It follows that $f(V_1)/V_1\geq f(V_2)/V_2$ on $[0,+\infty)$,
with strict inequality somewhere.
As a consequence, the first integral above is strictly positive. That is,
\eqref{V12(0)} holds. We then get a contradiction by arguing as in the previous
case.
\end{proof}

% \begin{proof}[Proof of Theorem \ref{thm:main} part (i)]
% 
% \end{proof}

We now derive some properties of positive, bounded solutions of
\eqref{y-stationary}.

\begin{proposition}\label{yproperties}
Assume that \eqref{hyp:f}, \eqref{hyp:g} hold and that $s\mapsto g(s)/s$
is nonincreasing. 
Let $S_*:=\inf\{S>0\ :\ g(S)\leq0\}$. Then, any positive, bounded
solution $(U,V)$ of \eqref{y-stationary} satisfies
$$S_*\leq\frac\nu\mu\solose\begin{cases}
                    S_*\leq U\leq\frac\nu\mu\\
                    V\leq1,
                   \end{cases}
\qquad
S_*\geq\frac\nu\mu\solose\begin{cases}
                    \frac\nu\mu\leq U\leq S_*\\
                    V\geq1.
                   \end{cases}
$$
\end{proposition}

\begin{proof}
Since $-dV''=f(V)$, $V$ cannot attain the value $1$ without being constant.
The following implications are then easily obtained:
$$V\equiv1\text{ on }[0,+\infty)\sse V(0)=1\sse V'(0)=0\sse g(U)=0\solose
S_*\leq U=\frac\nu\mu,$$
$$V<1\text{ on }[0,+\infty)\sse V(0)<1\sse V'(0)>0\sse g(U)<0\solose
S_*<U<\frac\nu\mu,$$
$$V>1\text{ on }[0,+\infty)\sse V(0)>1\sse V'(0)<0\sse g(U)>0\solose
\frac\nu\mu<U<S_*.$$
The result follows.
\end{proof}

It is not hard to construct examples
with $s\mapsto f(s)/s$ not nonincreasing for which the Liouville-type result
fails.
However, as shown by Proposition \ref{pro:y}, this hypothesis can be dropped
if $g$ satisfies some suitable conditions. The weaker sufficient condition we
are able to obtain is expressed in terms of the following quantity:
$$S_M:=\min\{s\geq0\ :\ g'\leq0\text{ in }[s,+\infty)\}.$$
Since $S_M=0$ when $g(s)=-\rho s$, the following result generalizes Proposition
\ref{pro:y}.
\begin{proposition}\label{pro:Liouville2}
Assume that \eqref{hyp:f}, \eqref{hyp:g} hold, that $s\mapsto g(s)/s$
is nonincreasing and that 
$$S_M\leq\frac{\nu+g(S_M)}\mu.$$
Then, problem \eqref{y-stationary} admits a unique positive, bounded solution. 
\end{proposition}

\begin{proof}
Let $(U,V)$ be a positive, bounded solution of \eqref{y-stationary}. We know
that $V(+\infty)=1$ and $V'(+\infty)=0$.
The same integration by parts as in the proof of Proposition \ref{pro:y}
yields $\theta(U)=0$, with
$$\theta(\sigma):=g^2(\sigma)+2d\int_1^{G(\sigma)}f(s)ds,
\qquad G(\sigma):=(\mu \sigma-g(\sigma))/\nu.$$
Thus, $U$ satisfies
\begin{equation}\label{sol}
\begin{cases}
(G(U)-1)g(U)\geq0\\
\theta(U)=0.
  \end{cases}
\end{equation}
It is not hard to show that, conversely, if $U>0$ satisfies \eqref{sol} then the
associated solution of \eqref{y-stationary} is positive and bounded.
The function $G$ satisfies $G(0)=0$,
$G(+\infty)=+\infty$ and it is strictly increasing in
the set where it is positive. Let $S_1>0$ be such that $G(S_1)=1$.
Set $S_*:=\inf\{S>0\ :\ g(S)\leq0\}$ and then call $a:=\min(S_1,S_*)$,
$b:=\max(S_1,S_*)$. 
For $U>b$, $G(U)>1$ and $g(U)\leq0$. Hence $U>b$ satisfies the first condition
in \eqref{sol} only if $g(U)=0$, but then $\theta(U)<0$ because $f<0$ in
$(1,+\infty)$. If $0<U<a$ then $(G(U)-1)g(U)<0$. This shows that \eqref{sol}
has no solution outside $[a,b]$.
Since $\theta(S_*)\leq0$ and $\theta(S_1)\geq0$,
the function $\theta$ vanishes somewhere on $[a,b]$. Moreover, $(G-1)g\geq0$
there. Therefore, \eqref{sol} admits solution on $[a,b]$. We conclude the proof
by showing that $\theta$ is strictly monotone on $[a,b]$. If $S_*<S_1$ then
$g^2$ is nondecreasing and $G$ is strictly increasing and smaller than $1$ on
$[a,b]$, whence $\theta$ is strictly increasing. Consider now the case
$S_*>S_1$. %Since $g(S_*)=0$, this case occurs if and only if $S_*>\nu/\mu$.
The function $G$ is strictly increasing and larger than $1$ on $[S_1,S_*]$.
On the other hand, by hypothesis,
$$G(S_M)=\frac{\mu S_M-g(S_M)}\nu\leq1.$$
Whence, $S_M\leq S_1$ and thus $g^2$ is decreasing on $[S_1,S_*]$.
We eventually infer that $\theta$ is strictly decreasing on $[S_1,S_*]$.
\end{proof}

\begin{remark}\label{rem:noLiouville}
Under the assumptions of Proposition \ref{yproperties}, if \eqref{y-stationary}
admits multiple positive, bounded solutions, then there exists a minimal one
$(\ul U,\ul V)$ among them. Moreover, $(\ul U,\ul V)$ satisfies $\ul
U\geq\nu/\mu$, $\ul V\geq1$ and it attracts solutions of \eqref{Cauchy}
starting below $(\nu/\mu,1)$. Indeed, if non-uniqueness occurs, then Proposition
\ref{pro:Liouville2} yields
$$S_*\geq S_M>\frac{\nu+g(S_M)}\mu\geq\frac\nu\mu.$$
Whence any positive, bounded
solution $(U,V)$ of \eqref{y-stationary} satisfies $U\geq\nu/\mu$ and $V\geq1$
due to Proposition \ref{yproperties}. Therefore, by comparison, if
$(u,v)$ is a solution of \eqref{Cauchy} with an initial datum
$(u_0,v_0)$ below $(\nu/\mu,1)$, then
$$\limsup_{t\to+\infty}u(t,x)\leq
U,\quad\limsup_{t\to+\infty}v(t,x,y)\leq V(y),$$
locally uniformly in $(x,y)\in\ol\Omega$.
On the other hand, if $(u_0,v_0)\not\equiv(0,0)$, Lemma \ref{lem:trapped}
implies that equality holds in the above expressions for one of these solutions
$(U,V)$. This is the minimal solution.
\end{remark}

%%%%%%%%%%%%%%%%%%%%%%%%%%%%%%%%%%%%%%%%%%%%

\section{Propagation}
\label{sec:velocity}
In this section we prove statements (ii) and (iii) of Theorem \ref{thm:main}. We
concentrate on
propagation to the right, since propagation to the left is obtained by
replacing
$q$ with $-q$. 
% So, we drop the $\pm$ superscripts and call $w_*$ the sought for
% spreading velocity.
The general plan of the proof is that of \cite{BRR2}. First, we look for plane
waves (or
exponential solutions) for the linearised system around $(0,0)$. These have a
critical velocity $w_*$. It then follows from the KPP assumption that solutions
of the nonlinear system spread at most with velocity $w_*$. To prove that they
spread at least with velocity $w_*$, we look at the problem in a wide strip
$\R\times(0,L)$ and we construct compactly supported sub-solutions using again
exponential solutions, but this time with complex exponents.
Finally, we establish the condition under which the critical wave velocity
is higher than the KPP speed. We do not repeat here all the arguments of
\cite{BRR2}, but we give details only of the points that are different.

%%%%%%%%%%%%%%%%%%%%%%%%%%%%%%%%%%%%%%%%%%%%

\subsection{Exponential solutions}\label{sec:exp}
Exponential solutions of the linearised system are looked under
the form
\begin{equation}
\label{exp}
(u(t,x),v(t,x,y))=e^{-\alpha(x-ct)}(1,\gamma e^{-\beta y}),
\end{equation}
with $c\geq0$, $\gamma>0$ and $\alpha,\beta\in\R$ (not
necessarily positive). Namely, we look for $c$, $\alpha$, $\beta$ and $\gamma$
satisfying
$$
\begin{cases}
c\alpha-D\alpha^2-q\alpha=\nu\gamma-\mu+g'(0)\\
c\alpha-d(\alpha^2+\beta^2)=f'(0)\\
d\gamma\beta=\mu-\nu\gamma.
\end{cases}
$$
The third equation yields $\gamma=\mu/(\nu+d\beta)$, which is positive iff
$\beta>-\nu/d$. This will be assumed without further reference.
Substituting we get
\begin{equation}\label{2exp}
\begin{cases}
-D\alpha^2+(c-q)\alpha=-\di\frac{d\mu\beta}{\nu+d\beta}+g'(0)\\
c\alpha-d(\alpha^2+\beta^2)=f'(0).
\end{cases}
\end{equation}
The first equation of \eqref{2exp} in the
unknown $\alpha$ has the roots
$$\alpha_D^\pm(c,\beta)=\frac1{2D}\left(c-q\pm\sqrt{
(c-q)^2+4D(\chi(\beta)-g'(0))}\right).$$
where the function $\chi$ is defined on $(-\nu/d,+\infty)$ by
$$\chi(\beta):=\frac{d\mu \beta}{\nu+d\beta}.$$
Since $\chi$ is strictly increasing and tends to $\mu$ at $+\infty$, we see
that
$\alpha_D^\pm(c,\beta)$ are
real iff 
\begin{equation}\label{ccond}
(c-q)^2>4D(g'(0)-\mu), 
\end{equation}
$$\beta\geq\underline\beta(c):=\chi^{-1}(g'(0)-(c-q)^2/4D).$$
Therefore, if $c$ satisfies \eqref{ccond}, the set of real solutions
$(\beta,\alpha)$ of the first equation of \eqref{2exp} is given
by $\Sigma(c)=\Sigma^-(c)\cup\Sigma^+(c)$, with
$$\Sigma^\pm(c):=\{(\beta,\alpha_D^\pm(c,\beta))\ :\ \beta\geq\ul\beta(c)\}.$$
This is a smooth curve with leftmost
point $(\ul\beta(c),(c-q)/2D)$. The second equation in \eqref{2exp} admits
real solutions iff $c\geq c_K$, where $c_K:=2\sqrt{df'(0)}$ is the
classical KPP critical speed for the second equation in \eqref{Cauchy}.
In the $(\beta,\alpha)$ plane, it represents the circle $\Gamma(c)$ of centre
$(0,c/2d)$ and radius
$$r(c)=\frac{\sqrt{c^2-c_K^2}}{2d}.$$
% The circle $\Gamma_{c,d}$ is
% the union of  the two half circles
% $$
% \Gamma_{c,d}=\{\alpha=\alpha_d^-(c,\beta)\}\cup\{\alpha=\alpha_d^+(c,\beta)\}
% $$
% where
% \begin{equation}
% \label{e4.12}
% \alpha_d^\pm(c,\beta)=\frac{c\pm\sqrt{c^2-c_K^2-4d^2\beta^2}}{2d}.
% \end{equation}
Let $\mc{S}(c)$ denote the closed set bounded from below by $\Sigma^-(c)$ and
from above by
$\Sigma^+(c)$ and let $\mc{G}(c)$ denote the closed disc
with boundary $\Gamma(c)$.
Exponential functions of the type \eqref{exp} are supersolutions of the
linearisation of \eqref{Cauchy} iff $(\beta,\alpha)\in \mc{S}(c)\cap
\mc{G}(c)$. It is easy to check that, for $c_K\leq c<c'$,
$\mc{G}(c)\subset\mc{G}(c')$.
Moreover, as $c\to+\infty$, the radius $r(c)$ of the disc $\mc{G}(c)$ diverges
and its bottom $(0,(c-\sqrt{c^2-c_K^2})/2d)$ tends to $(0,0)$.
Consequently
$$\bigcup_{c\geq c_K}\mc{G}(c)=\R\times(0,+\infty).$$
On the other hand,
the sets $\mc{S}(c)$ are not increasing with respect to $c$, but
the sets $\mc{S}(c)\cap(\R\times\R_+)$ are. Indeed, under condition
\eqref{ccond},
$\alpha_D^+(c,\beta)\geq0$ iff $c\geq q$ or $c\leq q$ and
$\chi(\beta)\geq g'(0)$, and in both cases
$$2D\partial_c\alpha_D^+(c,\beta)=1+\frac{c-q}{\sqrt{
(c-q)^2+4D(\chi(\beta)-g'(0))}}\geq0,$$
whereas $\alpha_D^-(c,\beta)\geq0$ iff $c\geq q$ and
$\chi(\beta)\leq g'(0)$, which yields
$$2D\partial_c\alpha_D^-(c,\beta)=1-\frac{c-q}{\sqrt{
(c-q)^2+4D(\chi(\beta)-g'(0))}}\leq0.$$
It follows that there exists $w_*\geq c_K$ such that
$\mc{S}(c)\cap\mc{G}(c)\neq\emptyset$ iff $c\geq
w_*$. Moreover, $w_*=c_K$ iff $(0,c_K/2d)\in\mc{S}(c_K)$.
Otherwise $\mc{S}(w_*)\cap\mc{G}(w_*)$ reduces to a point in $(0,+\infty)^2$,
denoted by $(\beta_*,\alpha_*)$. These notations will be kept in the following
sections.
% \begin{figure}[ht]
% \psfrag{b(c)}{\tiny$\ul\beta(c)$}
% \psfrag{c/2d}{\tiny$\frac c{2d}$}\psfrag{(c-q)/2D}{\tiny$\frac{c-q}{2D}$}
% \psfrag{a}{\scriptsize$\alpha$}\psfrag{b}{\scriptsize$\beta$}
% \psfrag{Gd}{\footnotesize$\Gamma(c)=\partial\mc{G}(c)$}
% \psfrag{G}{\footnotesize$\mc{G}(c)$}
% \psfrag{Si}{\footnotesize$\Sigma(c)=\partial\mc{S}(c)$}
% \psfrag{S}{\footnotesize$\mc{S}(c)$}
% \psfrag{(a,b)}{\tiny$(\alpha_*,\beta_*)$}
% \hoffset=.5cm
% \begin{center}
% \includegraphics[width=\textwidth]{c.eps}
% \caption{Case $\frac{c_K}{2d}\notin\mc{S}(c_K)$; $c<w_*$ (left), $c=w_*$
% (middle), $c>w_*$ (right).}
% \label{fig:c*}
% \end{center}
% \psfrag{cK/2d}{\tiny$\frac{c_K}{2d}$}
% \psfrag{a}{\scriptsize$\alpha$}\psfrag{b}{\scriptsize$\beta$}
% \psfrag{S}{\footnotesize$\mc{S}(c_K)$}
% \psfrag{a-}{\tiny$\alpha^-_D(c_K,0)$}
% \psfrag{a+}{\tiny$\alpha^+_D(c_K,0)$}
% \hoffset=.5cm
% \begin{center}
% \includegraphics[width=5cm]{cK.eps}
% \caption{Case $\frac{c_K}{2d}\in\mc{S}(c_K)$.}
% \label{fig:cK}
% \end{center}
% \end{figure}

% nuove figure
 \begin{figure}[ht]
 \centering
 \subfigure[$c<w_*$]
   {\includegraphics[width=4.5cm]{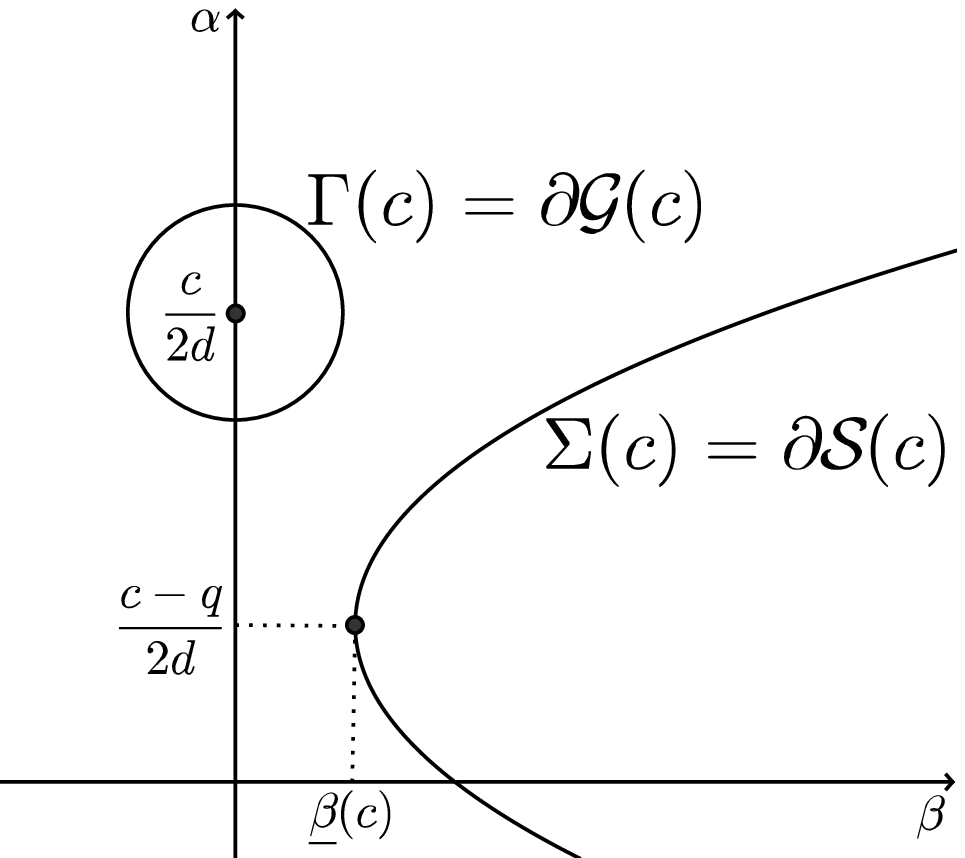}}
 \hspace{3mm}
 \subfigure[$c=w_*$]
   {\includegraphics[width=4.5cm]{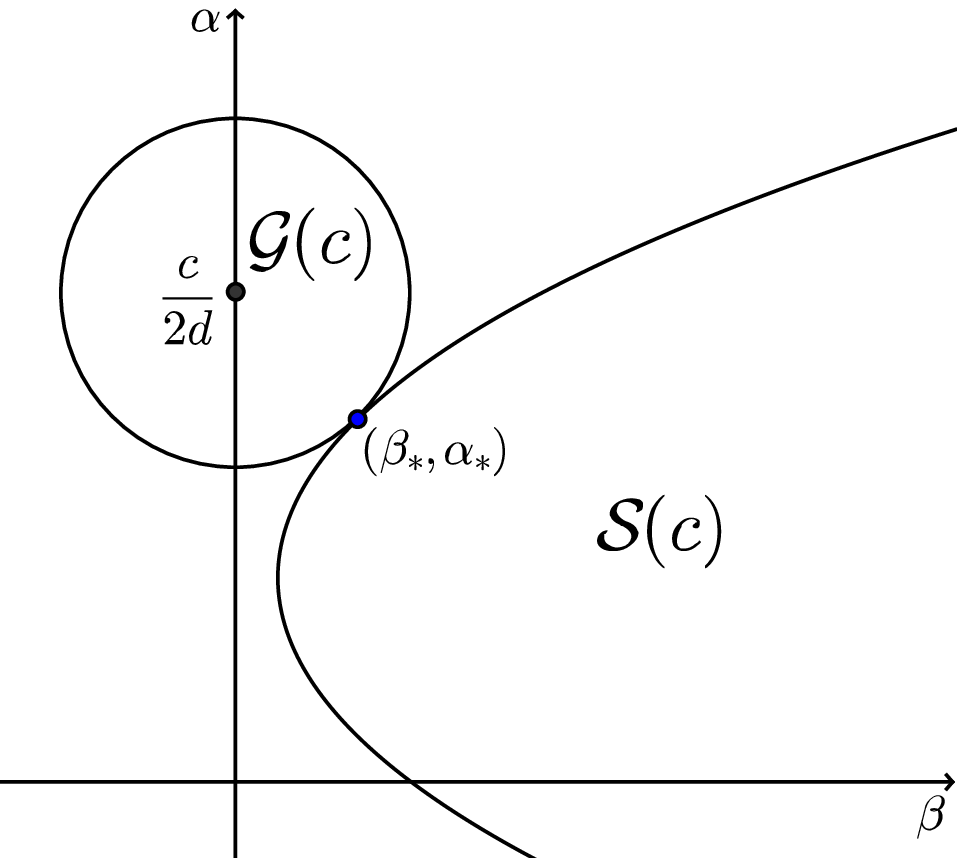}}
 \hspace{3mm}
 \subfigure[$c>w_*$]
   {\includegraphics[width=4.5cm]{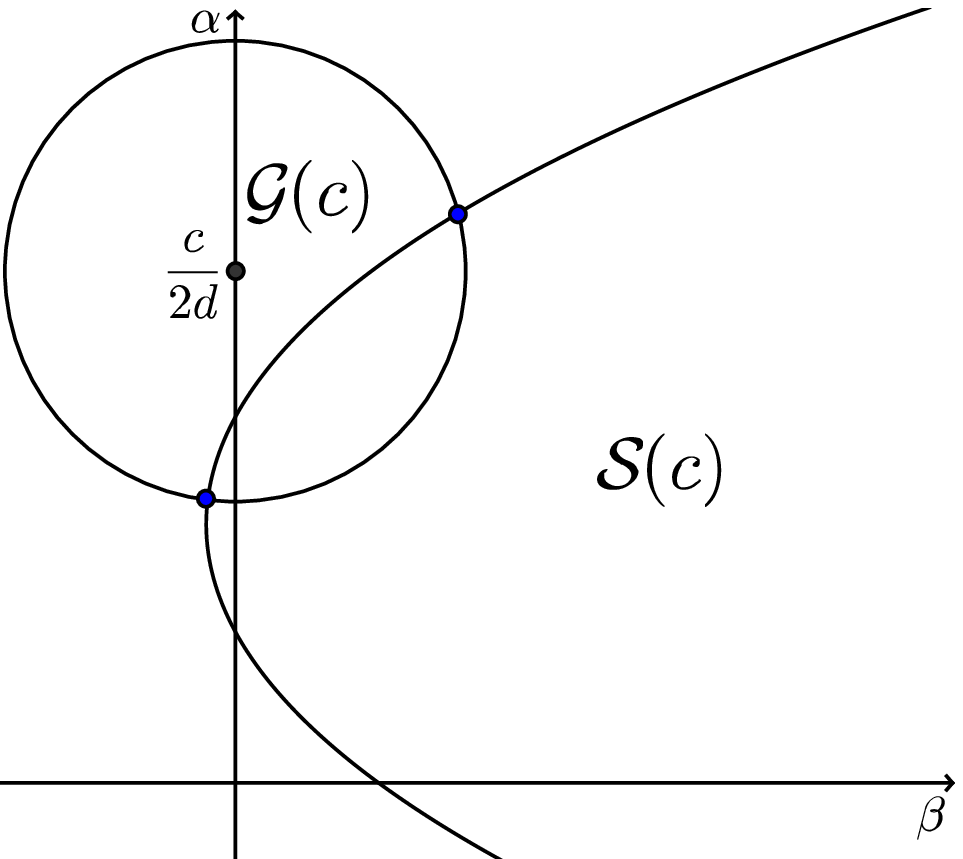}}
 \caption{Case $\frac{c_K}{2d}\notin\mc{S}(c_K)$.}
\bigskip
\includegraphics[width=5cm]{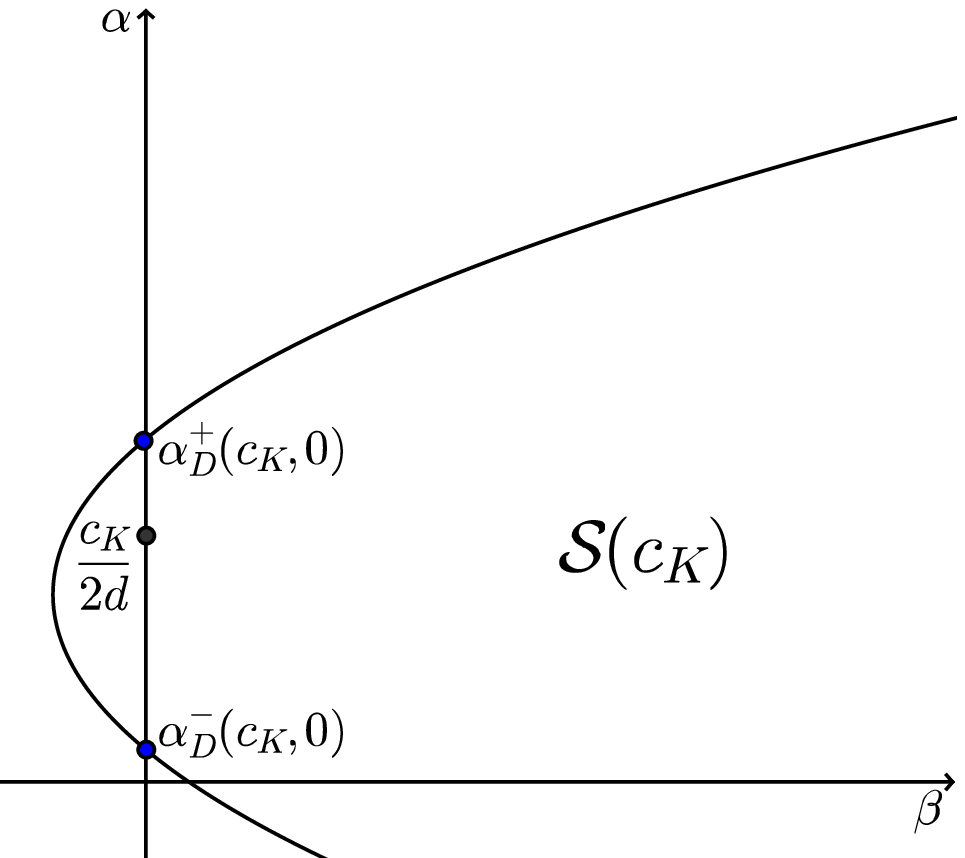}
\caption{Case $\frac{c_K}{2d}\in\mc{S}(c_K)$.}
\label{fig:cK}
 \end{figure}

%%%%%%%%%%%%%%%%%%%%%%%%%%%%%%%%%%%%%%%%%%%%

\subsection{Spreading}\label{sec:spreading}

\begin{proof}[Proof of Theorem \ref{thm:main} part (ii)]
Let $(u,v)$ be the solution of \eqref{Cauchy} starting from a nonnegative,
compactly supported initial datum $(u_0,v_0)\not\equiv(0,0)$.
We prove separately that $(u,v)$ spreads (towards right) at most and then at
least at
the critical velocity of plane waves $w_*$.

Step 1. {\em $(u,v)$ spreads at most with velocity $w_*$.}\\
The definition of $w_*$ implies the existence of an exponential supersolution of
the linearization of \eqref{Cauchy} (thus a supersolution of \eqref{Cauchy}
by \eqref{concave}) of the type \eqref{exp}, with $c=w_*$, $\gamma,\alpha>0$ and
$\beta\geq0$. Since, up to
translation in the direction $e_1$, this
supersolution is above $(u,v)$ at time $t=0$, the comparison principle yields 
% \begin{equation}\label{<w*} 
% \forall c>w_*,\quad 
% \lim_{t\to+\infty}(u(x+ct,t),v(x+ct,y,t))=(0,0),
% \end{equation}
% locally uniformly in $(x,y)\in\ol\Omega$.
$$\lim_{t\to+\infty}\sup_{\su{x>(w_*+\e)t}{y\geq0}}
|(u(x,t),v(x,y,t))|=0.$$
% This means that $(u,v)$ spreads at most with velocity $w_*$ in the direction
% $e_1$.
% 
% 
% By the definition of $w_*$, the set $\mc{S}(w_*)\cap\mc{G}(w_*)$ is
% nonempty. Taking $(\beta,\alpha)\in\mc{S}(w_*)\cap\mc{G}(w_*)$ and 
% $\gamma=\mu/(\nu+d\beta)$, it holds that $(e^{-\alpha(x-w_*t)},\gamma
% e^{-\alpha(x-w_*t)-\beta y})$
% is a supersolution of the linearization around $(0,0)$ of , and
% then of \eqref{Cauchy} by \eqref{concave}. Up to translation in the $e_1$
% direction, it is not restrictive to assume that  \eqref{<w*}.
% 
%
% We prove the statement by showing that if \eqref{B} holds then there
% exists an exponential supersolution moving at
% the speed $c_K$ in the direction $e_1$. We know that $(u,v)$ given by
% \eqref{exp} is a supersolution of (the
% linearisation of) \eqref{Cauchy} iff $(\beta,\alpha)\in \mc{S}(c)\cap
% \mc{G}(c)$. Thus, we need to show that \eqref{B} yields
% $\mc{S}(c_K)\cap
% \mc{G}(c_K)\neq\emptyset$. 

Step 2. {\em $(u,v)$ spreads at least with velocity $w_*$.}\\
We need to show that 
$$\forall \e,a>0,\quad\ \lim_{t\to+\infty}\sup_{\su{0\leq x\leq(w_*-\e)t}{
0\leq y<a}}|(u(x,t),v(x,y,t)-(U,V(y))|=0,$$
where $(U,V)$ is the unique positive, bounded, stationary solution of
\eqref{Cauchy}. We first show that 
\begin{equation}\label{local}
\forall \e>0,\ \exists c\in(w_*-\e,w_*), \quad
\lim_{t\to+\infty}(u(x+ct,t),v(x+ct,y,t))=(U,V(y)),
\end{equation}
locally uniformly in $(x,y)\in\ol\Omega$. Then we will conclude applying the
following lemma with $c^-=0$, $c^+=c$.
\begin{lemma}\label{lem:convexity}
Let $c^-<c^+$ be such that any nonnegative, bounded, solution 
$(u,v)\not\equiv(0,0)$ of \eqref{Cauchy} satisfies
$$\lim_{t\to+\infty}(u(x+c^\pm t,t),v(x+c^\pm t,y,t))=(U,V(y)),$$
locally uniformly in $(x,y)\in\ol\Omega$, with $U,V>0$. Then, such
solutions satisfy
% 
% 
% $A$ be a subset of $\R$ with the
% following property: any nonnegative, nontrivial supersolution 
% $(u,v)$ of \eqref{Cauchy} satisfies
$$\forall a>0,\quad\ \lim_{t\to+\infty}\sup_{\su{c^-t\leq x\leq c^+t}{
0\leq y<a}}|(u(x,t),v(x,y,t)-(U,V(y))|=0.$$
% Then the same property is fulfilled by the convex hull of $A$.
\end{lemma}
Let us postpone the proof of Lemma \ref{lem:convexity}.
In order to prove \eqref{local} it is sufficient
to derive the analogue of Lemma \ref{lem:trapped} for the problem in the frame
moving with speed $c$ in the direction $e_1$, for $c$ arbitrarily close to
$w_*$. Then one concludes using Theorem \ref{thm:y} (or Proposition \ref{pro:y}
in the case of \eqref{Cauchy-rho} under the only assumption \eqref{fKPP}).
The only argument in the proof of Lemma \ref{lem:trapped} which is modified by
the change of frame
is the existence of a compactly supported, stationary, strict subsolution $(\ul
U,\ul V)$ for the linearised problem. 
If $w_*=c_K$ then the construction of a function $\ul V$ such that $(0,\ul V)$
satisfies the desired properties is standard (see Lemma 6.2 of \cite{BRR2}).
In the case $w_*>c_K$, the bifurcation analysis of Lemma 6.1 of
\cite{BRR2} applies and provides the subsolution $(\ul U,\ul V)$ for $c$ close
enough to $w_*$. Since this is the core of the argument, for the sake of
completeness we sketch below the proof of that lemma.

One starts with penalizing $f'(0)$ and $g'(0)$ and considering the problem in
the strip $\R\times(0,L)$, with a Dirichlet condition on $\R\times\{L\}$:
\begin{equation}\label{strip}
\begin{cases}
-D\partial_{xx} u+(q-c)\partial_x u=\nu v(x,0,t)+(g'(0)-\delta+\mu)u & x\in\R\\
-d\Delta v-c\partial_x v=(f'(0)-\delta)v & (x,y)\in\R\times(0,L)\\
v(x,L)=0 & x\in\R\\
-d\partial_y v(x,0,t)=\mu u(x)-\nu v(x,0,t) & x\in\R.
\end{cases}
\end{equation}
The presence of $\delta<<1$ can be viewed as a perturbation of the terms
$f'(0)$, $g'(0)$; thus we can continue the discussion with $\delta=0$.
Exponential solutions are sought for in the form
$$e^{-\alpha x}(1,\gamma e^{-\beta y}+
\t\gamma e^{\beta y}).$$
They exist for $(\beta,\alpha)\in\Gamma(c)\cap\Sigma_L(c)$, where $\Gamma(c)$
is the same as in the previous section and $\Sigma_L(c)$ is the union of two
curves $\Sigma^\pm_L(c)$ parametrized by $\alpha=\alpha^\pm_L(c,\beta)$. The
functions $\alpha^\pm_L$ have the same monotonicity in $c$ as the functions
$\alpha^\pm$ defining $\Sigma(c)$; moreover, as $L\to\infty$,
$\alpha^\pm_L\to\alpha^\pm$ locally uniformly in $\beta>0$, together with their
derivatives.
Calling $\mc{S}_L(c)$ the set between $\Sigma^-_L(c)$ and $\Sigma^+_L(c)$, it
follows that $\mc{S}^L(c)\cap\mc{G}(c)\neq\emptyset$ iff $c\geq w_*^L$, with
$w_*^L\to w_*$ as $L\to\infty$. In particular, $w_*^L>c_K$ for $L$ large enough
(and $\delta$ small enough).
By using Rouch\'e's theorem one eventually finds, for $c<w_*^L$ close
enough to $w_*^L$, two values $\alpha,\beta\in\C\backslash\R$ such that the
associated exponential function satisfies \eqref{strip}. Using its real part,
one eventually obtain the compactly supported subsolution $(\ul U,\ul V)$ for
$c$ arbitrarily close to $w_*$.
\end{proof}

\begin{proof}[Proof of Lemma \ref{lem:convexity}]
Let $(u,v)$ be as in the statement of the lemma and fix $a>0$, $0<\e<U$.
Consider
the solution $(u^1,v^1)$ emerging from the initial datum $(\sup u,\sup v)$, and
a solution $(u^2,v^2)\not\equiv(0,0)$ with an
initial datum $(u^2_0,v^2_0)$ satisfying
$$0\leq u^2_0\leq U-\e,\qquad\supp\,u^2_0\subset[-1,1],
\qquad v^2_0\equiv0.$$
By hypothesis, for $a>0$, there exists $T>0$ such that
\begin{equation}\label{eq:sub-super}
\forall t\geq T,\ 0\leq y\leq a,\quad
|(u^i(c^-t,t),v^i(c^-t,y,t))-(U,V(y))|<\e.
\end{equation}
% Let $c_1,c_2\in A$ and, for $\lambda\in(0,1)$, consider
% the convex combination
% $c:=(1-\lambda)c_1+\lambda c_2$. 
% It is not restrictive to assume that
% $\lambda\geq1/2$. Fix $a>0$. 
Set $k:=\max(1,|c^-|T)$ and let $T'>0$ be such that
$$\forall t\geq T',\ |x|\leq k,\ 0\leq y\leq a,\quad
|(u(x+c^+t,t),v(x+c^+t,y,t))-(U,V(y))|<\e.$$
For $1/2\leq\lambda\leq1$, call $c:=(1-\lambda)c^-+\lambda c^+$.
Let $s\geq 2T'$. If $(1-\lambda)s\leq T$, writing $cs=
% (1-\lambda)sc^-+c^+\lambda s=
x+c^+t$ with $|x|\leq k$ and $t\geq T'$, we find that
$$\forall y\in[0,a],\quad|(u(ct,t),v(ct,y,t))-(U,V(y))|<\e.$$
Consider the case $(1-\lambda)s>T$. For $(x,y)\in\ol\O$, we see that 
$$u^2_0(x)\leq(U-\e)\mathbbm{1}_{[-1,1]}(x)\leq u(x+c^+\lambda s,\lambda s)\leq
u^1_0,\qquad
v^2_0\leq v(x+c^+\lambda s,y,\lambda s)\leq v^1_0,$$
and then the comparison principle yields
$$u^2(x+(1-\lambda)c^-s,(1-\lambda)s)\leq u(x+c s,s)\leq 
u^1(x+(1-\lambda)c^-s,(1-\lambda)s)$$
$$v^2(x+(1-\lambda)c^-s,y,(1-\lambda)s)\leq v(x+c s,y,s)\leq 
v^1(x+(1-\lambda)c^-s,y,(1-\lambda)s).$$
Whence, by \eqref{eq:sub-super},
$$\forall y\in[0,a],\quad
|(u(c^-s,s),v(c^-s,y,s))-(U,V(y))|<\e.$$
We have therefore shown that $(u,v)$ converges to $(U,V)$ as $t\to+\infty$
uniformly in the set 
$(c^-+c^+)t/2\leq x\leq c^+t$, $0\leq y\leq a$. The proof is completed by
exchanging the roles of $c^-$ and $c^+$.
\end{proof}

\subsection{Spreading velocity}

% \begin{proposition}
% It holds that
% % \begin{equation}\label{B}
% $$w_*=c_K\ \sse\ 
% \frac{D}d\leq2-\frac{g'(0)}{f'(0)}-\frac q{\sqrt{df'(0)}}.$$
% % \end{equation}
% % Then  $w_*=c_K$.
% \end{proposition}

\begin{proof}[Proof of Theorem \ref{thm:main} part (iii)]
% Notice that the statement is unaffected by the rescaling of the parameters by
% the factor $1/\nu$ performed in Section \ref{sec:exp}. 
We have shown in the previous section that the asymptotic speed of spreading
coincides with the critical velocity of plane waves $w_*$.
With the same notation as in Section \ref{sec:exp}, we know that
$w_*=c_K$ iff $(0,c_K/2d)\in\mc{S}(c_K)$, that
is, iff $\alpha_D^\pm(c_K,0)$ are real and satisfy
\begin{equation}\label{cKPP}
\alpha_D^-(c_K,0)\leq\frac{c_K}{2d}\leq\alpha_D^+(c_K,0).
\end{equation}
%Let us write for brief $\xi:=f'(0)$ and $\zeta:=g'(0)$.
% Since $\chi(0)=0$,
% the first condition in \eqref{cKPP} reads
% $$4D\zeta\leq4d\xi+q^2-4\sqrt{d\xi}q,$$
% that is
% \begin{equation}\label{A}
% \frac{D}d\leq\frac\xi\zeta+\frac{q-4\sqrt{d\xi}}{4d\zeta}q.
% \end{equation}
% We claim that \eqref{B} implies \eqref{A}. Indeed, computing the difference
% between the right hand sides of \eqref{A} and 
% \eqref{B} we get
% \[
% \begin{split}
% \frac\xi\zeta+\frac{q^2-4\sqrt{d\xi}q}{4d\zeta}-
% 2+\frac\zeta\xi+\frac q{\sqrt{d\xi}}
% &= \left(\sqrt{\frac\xi\zeta}-\sqrt{\frac\zeta\xi}\right)^2
% +\frac{\xi q-4\sqrt {d\xi^3}+4\sqrt{d\xi}\zeta}{4d\zeta\xi}q\\
% &= \frac1{4d\zeta\xi}\left(
% 4d(\xi-\zeta)^2+\xi q^2-4\sqrt{d\xi}(\xi-\zeta)q
% \right)\\
% &= \frac1{4d\zeta\xi}\left(
% 2\sqrt d(\xi-\zeta)-\sqrt\xi q\right)^2\\
% &\geq0.
% \end{split}
% \]
Requiring \eqref{cKPP} is equivalent to
$$\left|\frac{c_K}{2d}-\frac{c_K-q}{2D}\right|\leq\frac1{2D}\sqrt{
(c_K-q)^2-4Dg'(0)},$$
which, in turn, rewrites
$$\left(c_K\left(\frac Dd-1\right)+q\right)^2\leq(c_K-q)^2-4Dg'(0).$$
Notice that this condition automatically implies that $\alpha_D^\pm(c_K,0)$ are
real. Whence, recalling that $c_K=2\sqrt{d f'(0)}$, we find that \eqref{cKPP}
holds iff
$$4\frac Ddf'(0)+4\sqrt{\frac{f'(0)}d}\,q-8f'(0)\leq-4g'(0).$$
This concludes the proof.
\end{proof}

%%%%%%%%%%%%%%%%%%%%%%%%%%%%%%%%%%%%%%%%%%%%%%%%%%%%%%%%%%%%%%%%%%

\section{The large diffusion and transport limits of $w_*^\pm$}
\label{sec:limits}

We now prove Theorem \ref{thm:limits}. As before, we focus on the speed towards
right $w_*^+$, recalling that it coincides with the critical wave speed $w_*$
defined
in Section \ref{sec:exp}.
% 
% In this section, we investigate the behaviour of the critical speed $c_*$,
% introduced in Section \ref{sec:exp}, as $D$ and $q$ go separately to
% $+\infty$.
% We will show that, as a function of the variables $D$ and $q$ respectively,
% $c_*$ satisfies
% $$\lim_{D\to\infty}\frac{c_*}{\sqrt D}=h,\qquad
% \lim_{q\to+\infty}\frac{c_*}{q}=\begin{cases}
%                                    k & \text{if }g'(0)<\mu\\
%                                    1 & \text{if }g'(0)\geq\mu,\\
%                                   \end{cases}
% $$
% with $h>0$ independent of $q$ and $0<k<1$ independent of $D$.

\subsection{Large diffusion}

Replacing $c$ with $c=\sqrt D c$ and $\alpha$ with $\alpha/\sqrt D$ reduces
\eqref{2exp} to the following system:
\begin{equation}\label{D=1}
\begin{cases}
\di-\alpha^2+\left(c-\frac q{\sqrt D}\right)\alpha=-\frac{d\beta\mu}{\nu+d\beta}
+g'(0)\\
\di\alpha=\frac1c(f'(0)+d\beta^2)+\frac d{Dc}\alpha^2.
\end{cases}
\end{equation}
The first equation is satisfied for $(\beta,\alpha)\in \Sigma(c-q/\sqrt{D})$,
where $\Sigma=\Sigma^-\cup\Sigma^+$ is the curve defined in Section
\ref{sec:exp} with $D=1$ and $q=0$.
% Let $\mc{S(\.)}$ be the associated region as defined in Section \ref{sec:exp}.
The second equation is that of an ellipse $E_D(c)$ in the $(\beta,\alpha)$
plane, which is above the parabola $P(c)$ of equation
$\alpha=(f'(0)+d\beta^2)/c$. For $c$ close to $0$, $\Sigma(c)$ is below $P(c)$.
Then, increasing $c$, $\Sigma^+(c)$ moves upward while $P(c)$ moves downward and
tends to the $\beta$ axis as $c\to+\infty$. There exists then a positive value
$h$ such that $\Sigma(c)\cap P(c)\neq\emptyset$ if and only if $c\geq h$.
Since $\Sigma(c-q/\sqrt{D})$ and $E_D(c)$ converge locally
uniformly to $\Sigma(c)$ and $P(c)$ respectively as $D\to\infty$, and $\Sigma$
is bounded in the $\alpha$ direction, it follows that if $c<h$ then
\eqref{D=1} has no solution for $D$ large enough, whereas if $c>h$
then \eqref{D=1} has solution for $D$ large enough. Reverting to the original
variable, we
eventually infer that
$$\lim_{D\to\infty}\frac{w_*}{\sqrt D}=h.$$
% \begin{figure}[h]
% \psfrag{S}{\scriptsize$\mc{S}(c-\frac q{\sqrt{D}})$}
% \psfrag{f'}{\tiny$f'(0)/c$}
% \psfrag{a}{\scriptsize$\alpha$}\psfrag{b}{
% \scriptsize$\beta$}
% \psfrag{P}{\scriptsize$P(c)$}\psfrag{E}{\scriptsize$E_D(c)$}
% \psfrag{GD}{\scriptsize$\Sigma(c-\frac
% q{\sqrt{D}})$}\psfrag{S}{\scriptsize$\Sigma(c)$}
% \begin{center}
% \includegraphics[width=\textwidth]{Dinfty}
% \caption{$c$ close to $0$ (left), $c=h$ (middle), $c$ close to $\infty$
% (right).}
% \label{fig:Dinfty}
% \end{center}
% \end{figure}

% nuove figure
 \begin{figure}[ht]
 \centering
 \subfigure[$c<h$]
   {\includegraphics[width=4.5cm]{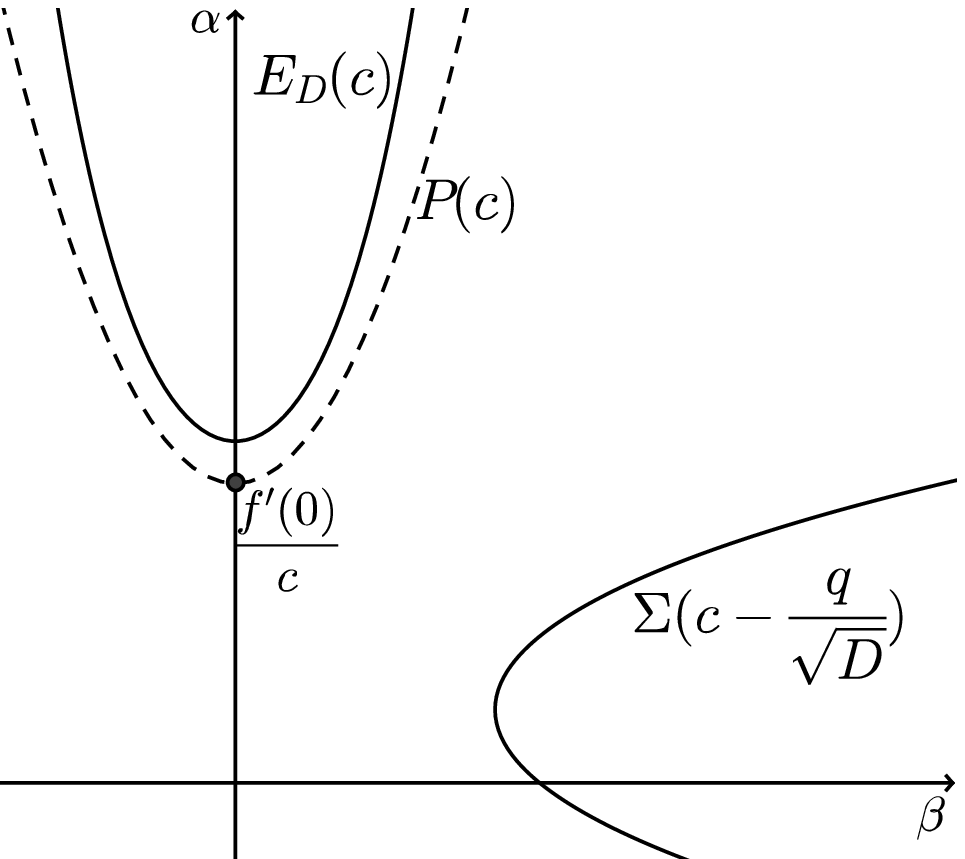}}
 \hspace{3mm}
 \subfigure[$c=h$]
   {\includegraphics[width=4.5cm]{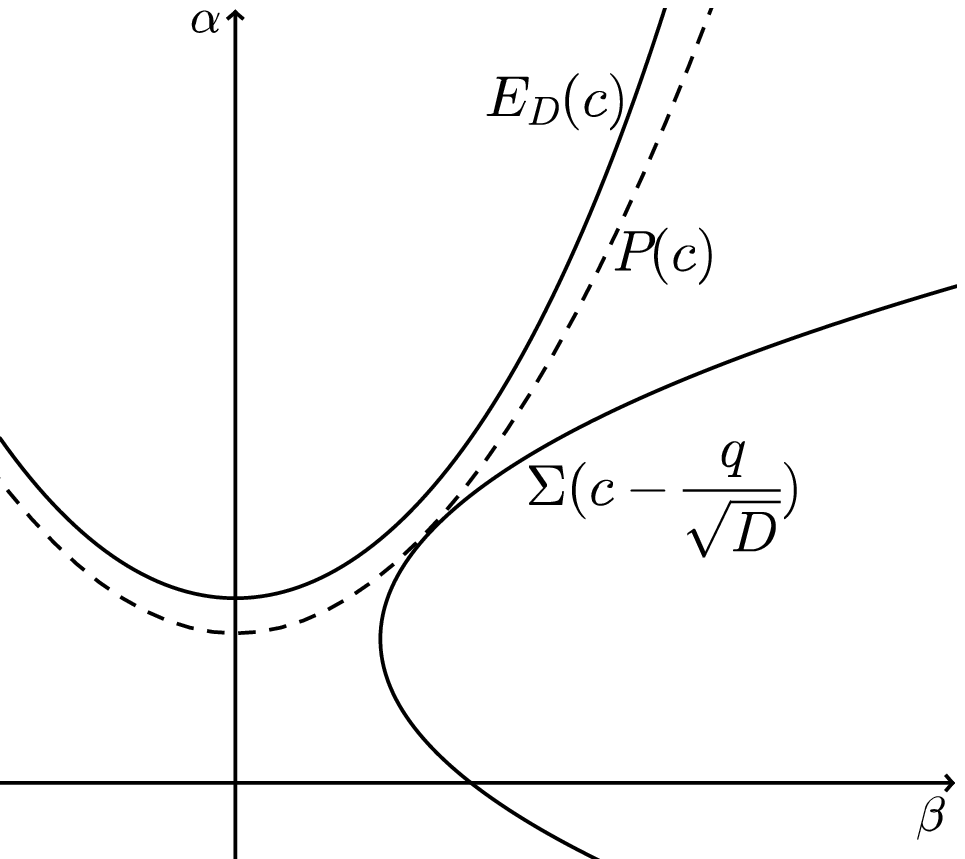}}
 \hspace{3mm}
 \subfigure[$c>h$]
   {\includegraphics[width=4.5cm]{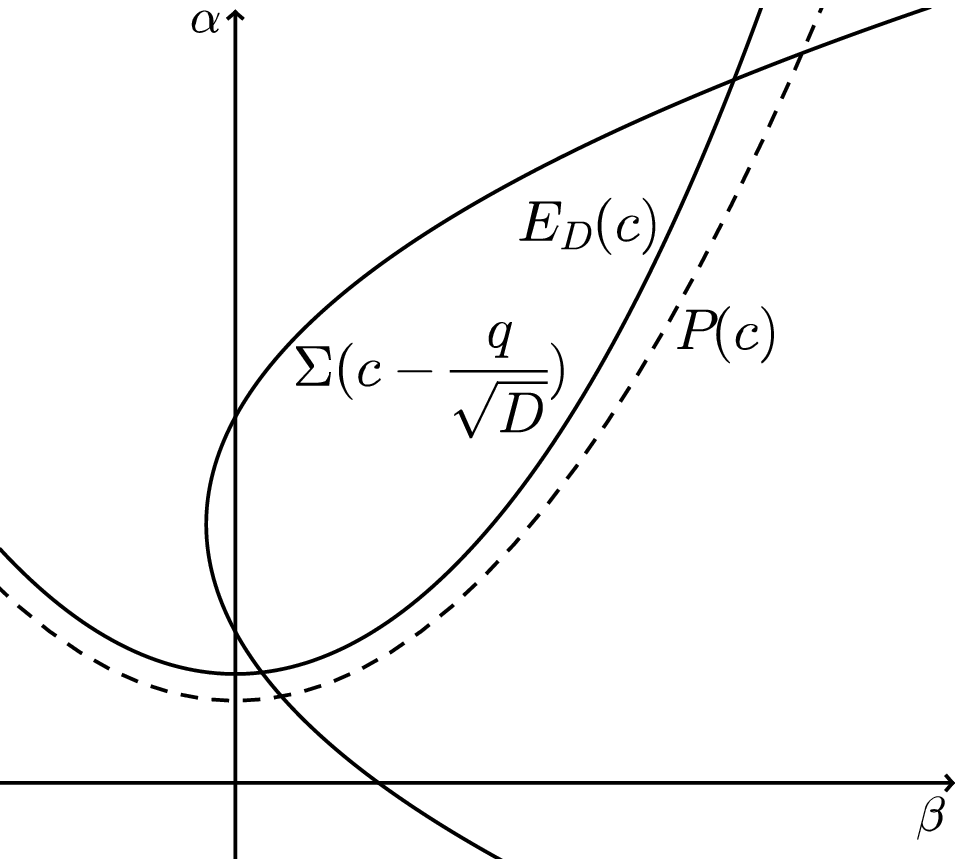}}
\caption{The asymptotics of $w_*$ as $D\to\infty$.}
 \end{figure}

\subsection{Large transport}
\label{sec:largeq}

We consider now $w_*$ as a function of $q$. We know that $w_*=c_K$ for $-q$
large enough. Let us investigate the behaviour as $q\to+\infty$.
Notice first that, as $q\to+\infty$, $\alpha_D^+(c,\beta)\to0$ locally
uniformly in $c$ and uniformly in $\beta\geq0$, from which we deduce that
$w_*\to+\infty$. Let us take $c=\kappa q$ with $\kappa>0$ and $q\to+\infty$. 
The set $\Gamma(\kappa q)$ intersect $\Sigma(\kappa q)$ if and only if there
exists
$\beta\geq0$ such that
% $\alpha_D^+(hq,+\infty)$ is larger than the $\alpha$ coordinate of the bottom
% of $\Gamma(hq)$. Namely,
% $$\frac1{2D}\left(q(h-1)+\sqrt{q^2(h-1)^2+4D(\mu-g'(0))}\right)>
% \frac{hq-\sqrt{h^2q^2-c_K^2}}{2d}.$$
% 
% Since the
% radius of $\Gamma(hq)$ diverges, it is not hard to check that
% $\Gamma(hq)\cap\Sigma(hq)$ is nonempty if and only if 
% Namely, if and only if there exists $\e>0$ such that, as
% $q\to+\infty$, 
% 
% 
% 
% The reverse inequality holds because
% $\Gamma(w_*)$ is below $\Sigma(w_*)$.
% This inequality holds if $h>1$. If $h<1$ it can be rewritten as
% $$\frac{\mu-g'(0)}{q(1-h)}\geq\frac{f'(0)}{hq}+o(1),$$
$$\frac1{2D}\left(q(\kappa-1)+\sqrt{q^2(\kappa-1)^2+4D(\chi(\beta)-g'(0))}
\right)\geq
\frac{\kappa q-\sqrt{\kappa^2q^2-c_K^2-4d^2\beta^2}}{2d}.$$
% That is,
% $$d(h-1)+d\sqrt{(h-1)^2+4D(\chi(\beta)-g'(0))q^{-2}}\geq
% Dh-D\sqrt{h^2-(c_K^2+4d^2\beta^2)q^{-2}}.$$
If $\kappa>1$, this inequality holds for any given $\beta\geq0$, provided $q$ is
large enough. This shows that $\limsup_{q\to+\infty}w_*/q\leq1$.
If $\kappa<1$, the inequality implies that, as $q\to+\infty$,
$$\frac{\chi(\beta)-g'(0)}{1-\kappa}+o(1)\geq
\frac{2f'(0)+2d\beta^2}{\kappa+\sqrt{\kappa^2-(c_K^2+4d^2\beta^2)q^{-2}}}.$$
In particular, $\beta$ cannot diverge as $q\to+\infty$, whence
$$\frac{\chi(\beta)-g'(0)}{1-\kappa}+o(1)\geq
\frac{f'(0)+d\beta^2}{\kappa}.$$
It follows that, if $g'(0)\geq\mu$ and $\kappa<1$,
$\Gamma(\kappa q)\cap\Sigma(\kappa q)=\emptyset$ for $q$ large enough. Instead,
if $g'(0)<\mu$, there is a threshold
value
$$k:=\left(1+\max_{\beta\geq0}\left(\frac{\chi(\beta)-g'(0)}{f'(0)+d\beta^2}
\right)\right)^{-1}<1$$
such that $\Gamma(\kappa q)\cap\Sigma(\kappa q)$ is empty
if $\kappa<k$ and $q$ is large enough and is nonempty if $\kappa>k$ and $q$
is large enough. Theorem \ref{thm:limits} follows.
\begin{remark}
 The computations in both the large diffusion and transport cases show that the
scale at which the limit of $w_*$ as $D,q\to+\infty$ is affected by both terms
is $D\sim q^2$. This was of course expected by dimensional considerations.
\end{remark}

\section*{Acknowledgement}
The research leading to these results has received funding from the European
Research Council under the European Union's Seventh Framework Programme
(FP/2007-2013) / ERC Grant Agreement n.321186 - ReaDi -Reaction-Diffusion
Equations, Propagation and Modelling. Part of this work was done while Henri
Berestycki was visiting the University of Chicago. He was also supported by an
NSF FRG grant DMS - 1065979. Luca Rossi was partially supported by the
Fondazione CaRiPaRo Project ``Nonlinear Partial Differential Equations:
models, analysis, and control-theoretic problems''.

\end{document}